\documentclass[12pt]{article}

\usepackage{amssymb,amsmath,amsthm}
\usepackage{hyperref}
\usepackage{epsfig}

\usepackage[utf8]{inputenc}

\newtheorem{theorem}{Theorem}[]
\newtheorem{lemma}[theorem]{Lemma}
\newtheorem{corollary}[theorem]{Corollary}
\newtheorem{observation}[theorem]{Observation}

\newtheorem{conjecture}[theorem]{Conjecture}
\newtheorem{claim}[theorem]{Claim}
\newtheorem{definition}[theorem]{Definition}

\newcommand\cref[1]{Corollary~\ref{cor:#1}}

\newcommand\HH{\mathcal{H}}
\newcommand\FF{\mathcal{F}}
\newcommand\cS{\mathcal{S}}

\begin{document}

\title{Coloring directed hypergraphs}

\author{Bal\'azs Keszegh\thanks{Alfréd Rényi Institute of Mathematics and Eötvös Loránd University, MTA-ELTE Lendület Combinatorial Geometry Research Group, Budapest, Hungary. Research supported by the Lend\"ulet program of the Hungarian Academy of Sciences (MTA), under the grant LP2017-19/2017, by the J\'anos Bolyai Research Scholarship of the Hungarian Academy of Sciences, by the National Research, Development and Innovation Office -- NKFIH under the grant K 132696 and FK 132060 and by the ÚNKP-20-5 New National Excellence Program of the Ministry for Innovation and Technology from the source of the National Research, Development and Innovation Fund.}}

\maketitle
\begin{abstract}
	Inspired by earlier results about proper and polychromatic coloring of hypergraphs, we investigate such colorings of directed hypergraphs, that is, hypergraphs in which the vertices of each hyperedge is partitioned into two parts, a tail and a head. 
	We present a conjecture of D. Pálvölgyi and the author, which states that directed hypergraphs with a certain restriction on their pairwise intersections can be colored with two colors. Besides other contributions, our main result is a proof of this conjecture for $3$-uniform directed hypergraphs. This result can be phrased equivalently such that if a $3$-uniform directed hypergraph avoids a certain directed hypergraph with two hyperedges, then it admits a proper $2$-coloring. Previously, only extremal problems regarding the maximum number of edges of directed hypergraphs that avoid a certain hyperedge were studied.
\end{abstract}

\section{Introduction}

Graphs are perhaps the most important objects of combinatorics and coloring problems are among the most interesting and well-studied problems concerning graphs. Hypergraphs are natural generalizations that have many applications and are also widely studied\footnote{A hypergraph $\HH$ is defined by a pair $(V,E)$ where $V$ is a finite set of vertices and $E$, the family of its hyperedges, is a family of its subsets.}. Coloring problems about hypergraphs also receive a lot of interest, among which the case of proper $2$-colorings\footnote{A $c$-coloring of the vertices of a hypergraph is proper if every hyperedge contains vertices with at least two different colors.} has perhaps the longest history. Proper $2$-colorability is also called Property $B$. In the last few years polychromatic $c$-colorings\footnote{A $c$-coloring of the vertices of a hypergraph is $c$-polychromatic if every hyperedge contains vertices with all $c$ colors.} of hypergraphs gained attention, especially for hypergraphs that can be defined by geometric objects in natural ways (e.g., by containment relations). 
For more about coloring hypergraphs we refer the reader to \cite{berge,tuza}, for more about coloring geometric hypergraphs we refer the reader to the up-to-date database \cite{cogezoo} including the references therein.

While it is $NP$-complete to decide if a graph has chromatic number $c\ge 3$, the $c=2$ case, i.e., recognizing bipartite graphs, can be done easily in polynomial time. On the other hand, already for $3$-uniform hypergraphs it is $NP$-complete to decide if they admit a proper $2$-coloring. Thus, finding natural conditions that imply proper $2$-colorability is an interesting endeavour, already for $3$-uniform hypergraphs, see, e.g., \cite{lovaszcikk} for some early results in this direction.
Using as a starting point a well-known property that guarantees proper $2$-colorability, the aim of this paper is to give natural new conditions on the intersections of hyperedges that guarantee proper $2$-colorability or polychromatic $c$-colorability. We define a new condition, which we call Property S. We conjecture that Property S implies Property B, i.e., the proper $2$-colorability of the hypergraph. Our main result is a proof of this conjecture for $3$-uniform hypergraphs. We also show that it holds for linear hypergraphs.

Along the way we also prove some results which give upper bounds on the number of hyperedges assuming Property S or other similar properties.

\subsection{Restrictions on the intersection sizes}

A classic exercise from the book of Lov\'asz (\cite{lovasz07}, Problem 13.33), is the following:

\begin{theorem}\label{thm:lovasz}
	Given a hypergraph $\HH$ such that every pair of hyperedges has an empty intersection or intersects in at least two vertices, then $\HH$ admits a proper $2$-coloring.
\end{theorem}

This can be generalized to give a condition that guarantees a polychromatic $c$-coloring:

\begin{theorem}\label{thm:poly}
	Given a positive integer $c\ge 2$ and a hypergraph $\HH$ such that every set of at most $c$ hyperedges has an empty intersection or intersects in at least $c$ vertices, then $\HH$ admits a polychromatic $c$-coloring.
\end{theorem}

Note that the assumption of Theorem \ref{thm:poly} for $c=1$ implies that the hyperedges of $\HH$ all have size at least $c$.  The following more general theorem implies immediately Theorem \ref{thm:poly} by setting $\HH_i=\HH$ for all $i$:

\begin{theorem}\label{thm:polyassim}\cite{erdosp2}
	Given hypergraphs $\HH_1,\dots, \HH_c$ on the same vertex set such that for every $i$ ($1\le i\le c$) every set of $i$ hyperedges from different hypergraphs has an empty intersection or intersects in at least $i$ vertices, then the vertices admit a $c$-coloring such that for every $j$, every $H\in H_j$ contains a vertex with color $j$.
\end{theorem}

P. Erd\H{o}s has considered this problem first for $c=2$ then later proved the general case by induction on the number of colors \cite{erdosp1,erdosp2}. His proof can be considered to be the generalization of the proof of Theorem \ref{thm:lovasz} as presented in \cite{lovasz07}. Here we present a new proof based on a recoloring idea.

\subsection{Restrictions on the 1-intersection graph}\label{sec:gyarfas}
A different possible generalization of Theorem \ref{thm:lovasz} is the following. We define the 1-intersection graph of $\HH$ as a graph that has vertices corresponding to the hyperedges of $\HH$ and two vertices are connected by an edge if and only if the corresponding hyperedges intersect in exactly one vertex. Theorem \ref{thm:lovasz} states that if the 1-intersection graph is empty then the hypergraph is $2$-colorable. Gyárfás et al. \cite{gyarfas} generalized this by showing that if $\HH$ is $3$-uniform then it is $2$-colorable even when the 1-intersection graph is required only to be bipartite. More generally they showed that the chromatic number of a $3$-uniform hypergraph is at most the chromatic number of its 1-intersection graph (assuming that the latter is at least two). It is an open problem if this holds for $k$-uniform hypergraphs with $k>3$. It may even be true for any hypergraph in which all edges have size at least $3$. Yet, already the case if a $4$-uniform hypergraph with bipartite 1-intersection graph is $2$-colorable is open.

\subsection{Directed hypergraphs}
While the previous generalization assumed some global properties of the 1-intersection graph, our aim is to strengthen Theorem \ref{thm:lovasz} by having only local constraints. Our constraints need that some vertices of the hyperedges are deemed special and then the intersection of two hyperedges is restricted with respect to these special vertices. One such notion of hyperedges with two kinds of vertices (special and non-special) is the notion of \emph{set-pairs} which found many applications and generalizations in extremal combinatorics starting with the famous result of Bollob\'as about intersecting set-pair systems \cite{Bollobas1965}. 

A more recent notion, generalizing the notion of directed graphs, is that of \emph{directed hypergraphs} \cite{gallo}: a directed hypergraph is a hypergraph in which every hyperedge of a hypergraph is split into two parts, a tail and a head.\footnote{This notion should not be confused with other notions of directed hypergraphs, where there is a total order on the vertex set of each hyperedge, e.g., in \cite{simonovits}.} A vertex in the tail (resp. head) of a hyperedge is called a tail-vertex (resp. head-vertex) of the hyperedge. We allow a hyperedge to have an empty tail or head. Similar to intersecting set-pairs, mostly extremal problems were considered in this setting as well. Usually we are interested in the maximum number of hyperedges in a directed hypergraph that avoids some fixed directed hypergraph as a subhypergraph.

Directed hypergraphs with only hyperedges of size $3$ having a tail of size two and head of size one are called \emph{$2\rightarrow 1$ hypergraphs} and received the most attention (see, e.g., \cite{turangy,21cameron}). A $2\rightarrow 1$ hypergraph is said to be \emph{oriented} if it does not contain two hyperedges on the same $3$ vertices. 
Let $ex(n,\FF)$ denote the maximal number of hyperedges in a $2\rightarrow 1$ hypergraph on $n$ vertices avoiding $\FF$. Let $ex_o(n,\FF)$ denote the maximal number of hyperedges in an {\it oriented} $2\rightarrow 1$ hypergraph on $n$ vertices avoiding $\FF$. 

While we also present some extremal results, our focus is on showing that if a directed hypergraph avoids certain directed subhypergraphs with two hyperedges then it is proper $2$-colorable.


\begin{claim}\label{claim:specboth}
	Given a directed hypergraph $\HH$ such that every hyperedge has at least one tail-vertex and for every pair of hyperedges $H_1,H_2\in \HH$, if $H_1\cap H_2=\{v\}$ then $v$ must be a head-vertex of both hyperedges. Then $\HH$ admits a proper $2$-coloring.
\end{claim}

It seems that if we require that when two hyperedges intersect in one vertex $v$, then $v$ must be a head-vertex of only at least one of these two hyperedges instead of both, then the following might be true, conjectured by D. P\'alv\"olgyi and the author:

\begin{conjecture}\label{conj:majo}
	Given a directed hypergraph $\HH$ such that in every hyperedge the number of head-vertices is less than the number of tail-vertices and for every pair of hyperedges $H_1,H_2\in \HH$, if $H_1\cap H_2=\{v\}$ then $v$ is a head-vertex in at least one of the hyperedges. Then $\HH$ admits a proper $2$-coloring.
\end{conjecture}

Observe that Conjecture \ref{conj:majo} is a generalization of Theorem \ref{thm:lovasz}, as we aimed for. Let us call Property S the assumption of this conjecture, that is:

\begin{definition}
	Given a directed hypergraph $\HH$. If in every hyperedge of $\HH$ the number of head-vertices is less than the number of tail-vertices and $\HH$ does not contain two hyperedges intersecting in a single vertex which is a tail-vertex in both of them, then we say that $\HH$ has \emph{Property S}.
\end{definition}

Thus with this notation the conjecture claims that Property S implies Property B. In Section \ref{sec:shortproofs} we give an example showing that in the conjecture the assumption on the number of head-vertices in a hyperedge cannot be increased, i.e. we cannot change `less than' to `at most'. While Conjecture \ref{conj:majo} is open in general, we prove that it holds in special cases. Before showing these cases, first we prove that this property does decrease the maximum number of hyperedges:

\begin{claim}\label{claim:size}
	Given a $k$-uniform directed hypergraph $\HH$ with Property $S$, i.e. every hyperedge of $\HH$ has less head-vertices than tail-vertices and for every pair of hyperedges $H_1,H_2\in \HH$, if $H_1\cap H_2=\{v\}$ then $v$ is a head-vertex in at least one of the hyperedges. Then $\HH$ has $O(n^{k-1})$ hyperedges.
\end{claim}

Note that Claim \ref{claim:size} is optimal as $\HH$ can be the directed hypergraph containing all hyperedges that contain some fixed vertex $v$ which is set to be the head-vertex of every hyperedge.
	
We can show that Conjecture \ref{conj:majo} holds for linear hypergraphs:\footnote{A hypergraph is \emph{linear} if every pair of hyperedges intersects in at most one vertex. This is a widely studied class of hypergraphs, e.g., the well-known Erd\H{o}s-Faber-Lov\'asz conjecture which was (essentially) solved only recently \cite{efl-conj} states that the vertices of any $k$-uniform linear hypergraph with $k$ hyperedges admits a polychromatic $k$-coloring. A directed hypergraph is said to be linear if the underlying non-directed hypergraph is linear.}

\begin{theorem}\label{thm:linear}
	Given a directed hypergraph $\HH$ with Property S, i.e., every hyperedge of $\HH$ has less head-vertices than tail-vertices and for every pair of hyperedges $H_1,H_2\in \HH$, if $H_1\cap H_2=\{v\}$ then $v$ is a head-vertex in at least one of the hyperedges. If $\HH$ is linear, then $\HH$ admits a proper $2$-coloring.
\end{theorem}

\smallskip
Our main result is a proof of Conjecture \ref{conj:majo} for the special case of $3$-uniform hypergraphs. Compare this with the problem about global restrictions on 1-intersections \cite{gyarfas} from Section \ref{sec:gyarfas}, where also the $3$-uniform case is the most general case solved so far. Note that Property S remains true in a directed hypergraph if in some hyperedge we change a tail-vertex to be a head-vertex as far as there are still less head-vertices than tail-vertices in the hyperedge. In particular, when considering Conjecture \ref{conj:majo} for $3$-uniform directed hypergraphs, it is enough to consider the case when every hyperedge has exactly one head-vertex, that is, it is a $2\rightarrow 1$ hypergraph. Conversely, in Property S the restriction on the number of head-vertices automatically holds for $2\rightarrow 1$ hypergraphs. Thus, the following theorem is equivalent to Conjecture \ref{conj:majo} for $3$-uniform hypergraphs:

\begin{theorem}\label{thm:3unif}
	Given a $2\rightarrow 1$ hypergraph $\HH$ with Property S, i.e.,
	 for every pair of hyperedges $H_1,H_2\in \HH$, if $H_1\cap H_2=\{v\}$ then $v$ is a head-vertex in at least one of the hyperedges. Then $\HH$ admits a proper $2$-coloring.
\end{theorem}

Let $\cS$ be the $2\rightarrow 1$ hypergraph with vertex set $\{1,2,3,4,5\}$ and with hyperedge set $\{12 \rightarrow 3,14 \rightarrow 5\}$.\footnote{In a $2\rightarrow 1$ hypergraph a hyperedge with tail-vertices $a$ and $b$ and head-vertex $c$ is written as $ab\rightarrow c$.} Notice that a $2\rightarrow 1$ hypergraph has Property S if and only if it avoids $\cS$. Therefore, Theorem \ref{thm:3unif} states that if a $2\rightarrow 1$ hypergraph avoids $\cS$ then it admits a proper $2$-coloring. Note that from Claim \ref{claim:size} it follows that $ex(n,\cS)=O(n^2)$. In \cite{21cameron} for all possible $2\rightarrow 1$ hypergraphs with two edges the corresponding extremal problem was considered and about $\cS$ it was shown that $ex(n,\cS)=\binom{n+1}{2}-3$ and $ex_0(n,\cS)={\lfloor \frac{n}{2}}\rfloor(n-2)$.

Let us finish with answering also an extremal problem about $2\rightarrow 1$ hypergraphs. The $2\rightarrow 1$ hypergraph $\FF_0=\{12\rightarrow 3,13\rightarrow 2,23\rightarrow 1\}$ is an example for which we have that $ex(n,\FF_0)-ex_o(n,\FF_0)=\Omega(n^3)$. Note that $\FF_0$ is not oriented and no oriented $2\rightarrow 1$ hypergraph with the same property has been known. Answering a question of A. Cameron \cite{21cameron}, here we show that a diluted version of $\FF_0$ is an oriented $2\rightarrow 1$ hypergraph that has this property:

\begin{theorem}\label{thm:21ex}
	Let $\FF$ be the oriented $2\rightarrow 1$ hypergraph with vertex set $\{1,2,3,4,5\}$ and with hyperedge set $\{12\rightarrow 3,13\rightarrow 4,23 \rightarrow 5,14 \rightarrow 2,25 \rightarrow 1\}$. 
	Then $ex_o(n,\FF)=\binom{n}{3}$ while $ex(n,\FF)\ge 2\binom{n}{3}$.
\end{theorem}

\section{The short proofs}\label{sec:shortproofs}

%
%

A very nice result of Berge \cite{berge} gives a sufficient condition for a hypergraph to admit a polychromatic $c$-coloring. It states that if every induced subhypergraph\footnote{Given a subset $S\subset V$ of the vertices of $\HH(V,E)$, the subhypergraph induced by $S$ has vertex set $S$ and hyperedge set $\{S\cap H:H\in E, |S\cap H|\ge 2\}$.} of a hypergraph $\HH$ is proper $2$-colorable then we can $c$-color the vertices of $\HH$ such that every hyperedge $H$ contains vertices with $\min\{|H|,c\}$ different colors. 
This recoloring idea can be used to give a different proof of Theorem \ref{thm:lovasz} and it can also be used to prove its generalization to polychromatic $c$-colorings:

\begin{proof}[Proof of Theorem \ref{thm:polyassim}]
	The colors are denoted by the numbers of $[c]$. Note that a set of vertices is allowed to be a hyperedge of $\HH_i$ for multiple $i$'s.
	
	Given a coloring, we call a hypergraph $H$ good if it contains a vertex with the required color (that is, with color $i$ if $H\in \HH_i$). We call a coloring good if all its hyperedges are good. We start with an arbitrary coloring of the vertices of $\HH$. If this is a good $c$-coloring then we are done. From now on we assume that this is not the case and we shall improve the coloring. By improvement we mean that the number of hyperedges that do not have a vertex with the required color decreases in each step, call this function $f$.
	
	There is some color, wlog. $1$, such that there is an edge $H_1\in \HH_1$ that misses color $1$. Take an arbitrary vertex $v_2$ of $H_1$, wlog. it has color $2$. Recoloring $v_2$ to color $1$ makes $H_1$ good. Then $f$ decreases unless there is a hyperedge $H_2\in \HH_2$ containing $v_2$ and no other vertex with color $2$ (in which case we do not recolor $v_2$). Now $|H_1\cap H_2|\ge 2$ by the assumption of the theorem, thus there is another vertex $v_3$ in $H_1\cap H_2$, which has some color different from $1$ and $2$, wlog. it is colored $3$. Similarly as before, recoloring $v_3$ to color $1$ makes $H_1$ good. Then $f$ decreases unless there is a hyperedge $H_3\in \HH_3$ containing $v_3$ and no other vertex with color $3$ (in which case we do not recolor $v_3$). In a general step we have already found  hyperedges $H_2\in \HH_2,\dots ,H_i\in \HH_i$ and vertices $v_2\in H_1\cap H_2,v_3\in H_1\cap H_2\cap H_3\dots ,v_i\in H_1\cap H_1\cap H_2\cap\dots \cap H_i$ with the property that for every $j$ ($2\le j\le i$) $v_j$ is the only vertex of $H_j$ with color $j$. Now for the intersection $X_i=H_1\cap H_2\cap\dots ,\cap H_i$ as it is none-empty, we have $|X_i|\ge i$ by the assumption of the theorem, thus there exists a vertex $v_{i+1}\in X_i\setminus \{v_2,\dots ,v_{i}\}$. Notice that the color of $v_{i+1}$ is not in $[i]$, wlog. it has color $i+1$. Now recoloring $v_{i+1}$ to color $1$ makes $H_1$ good. Then $f$ decreases unless there is a hyperedge $H_{i+1}\in \HH_{i+1}$ containing $v_{i+1}$, in which case we do not recolor $v_{i+1}$ but continue this process (note that $v_{i+1}\in H_1\cap H_1\cap H_2\cap\dots \cap H_i\cap H_{i+1}$).

	By the end of the process we either can recolor some vertex of $H_1$ to decrease $f$ or find a vertex $v_{c+1}$ that must have color not in $[c]$, which is a contradiction. Thus there was a valid recoloring. We can repeat this whole process until $f$ reaches $0$, which means that we got a good $c$-coloring.
\end{proof}

\begin{proof}[Proof of Claim \ref{claim:specboth}]
	The proof is also based on the recoloring idea used in the previous proof.
	We start with an arbitrary $2$-coloring of the vertices of $\HH$. If this is a proper $c$-coloring then we are done. From now on we assume that this is not the case and we shall improve the coloring. By improvement we mean that the number of hyperedges that are monochromatic decreases in each step, call this function $f$.
	
	If a hyperedge $H$ is monochromatic, then recolor one of its tail-vertices $v$ to the other color. Assume this made another hyperedge $H'$ monochromatic. This means that $v=H\cap H'$. As $v$ is a tail-vertex in $H$, this is a contradiction. Thus, this recoloring decreased $f$. Repeating this process until $f$ reaches $0$ we get a proper $2$-coloring of $\HH$, as required.	
\end{proof}

\noindent{\bf Construction.}
Next we construct an example showing that in Conjecture \ref{conj:majo} the assumption on the number of head-vertices in a hyperedge cannot be increased, i.e. we cannot change `less than' to `at most'. To see this, take a vertex set $V$ with $3k-3$ vertices split into three equal parts, $V_0,V_1,V_2$, all having size $k-1$. Now let the hyperedges of $\HH$ with vertex set $V $ be the sets which for some $i\in \{0,1,2\}$ contain $k$ vertices from $V_i$ and $k$ vertices from $V_{i+1}$ (indices are modul0 $3$). For such a hyperedge its vertices in $V_i$ are its head-vertices. Thus we got a directed hypergraph in which for each hyperedge half of its vertices is a head-vertex, if two hyperedges intersect then their intersection contains at least one vertex which is a head-vertex in one of them. We claim that $\HH$ does not admit a proper $2$-coloring. Assuming on the contrary that there is a proper $2$-coloring, then there is an index $i$ such that both in $V_i$ and $V_{i+1}$ the majority color is the same in this coloring, say red. That is, in both $V_i$ and $V_{i+1}$ there are $k$ red vertices. Thus there is a hyperedge on $\HH$ using only these vertices, which is then monochromatic, a contradiction.

\begin{proof}[Proof of Claim \ref{claim:size}]
	We count the pairs $(H,v)$ where $v\in H$ is a tail-vertex of $H$.
	Fix any vertex $v_0$ and count such pairs $(H,v_0)$. If there is at least one, $(H_0,v_0)$, then for every other such pair $(H',v_0)$ we must have $|H'\cap H_0|\ge 2$ as otherwise they would intersect only in $v$ which is a tail-vertex in both $H_0$ and $H'$, contradicting our assumption. The number of hyperedges that intersect $H_0$ in at least two vertices is $O(n^{k-2})$ and thus this also upper bounds the number of pairs $(H,v_0)$. Summing over all $n$ vertices, we have at most $O(n^{k-1})$ pairs $(H,v)$. Thus, the size of the hypergraph is also upper bounded by $O(n^{k-1})$ as every hyperedge appears in at least one pair.
\end{proof}

\begin{proof}[Proof of Theorem \ref{thm:linear}]
	We prove by induction on the number of vertices. The theorem holds when the hypergraph has no hyperedges or we have $2$ vertices. From now on $\HH$ has at least $3$ vertices and at least one hyperedge. We prove by induction on the number of vertices of $\HH$.
	
	First, if there is a vertex with degree zero then we are done by coloring this vertex arbitrarily and applying induction on the hypergraph with the remaining vertices as its vertex set. Thus we can assume that there is no such vertex.
	
	Now, if we can find a vertex with degree one then we are done by induction. Indeed, in this case we can delete this vertex and the unique hyperedge containing it and color the remaining hypergraph by induction. Afterwards, we can put back the hyperedge and the vertex, the latter we color such that this hyperedge is not monochromatic.
	
	We are left to find a vertex with degree one. Let $G$ be the bipartite graph in which one part $A$ corresponds to the vertices of $\HH$, the other part $B$ corresponds to the hyperedges of $\HH$ and the edges correspond to incidences in the hypergraph. Direct the edges of $G$ such that the edge $(H,v)$ is directed towards $v$ if and only if $v$ is a tail-vertex in $H$. Let $n_k$ be the number of vertices in $A$ with (undirected) degree $k$ (for $k\ge 1$). Observe that $\sum_k n_k=n$.
	
	Let $a$ (resp. $b$) be the number of edges directed towards $A$ (resp. $B$). 
	
	First, observe that the outdegree of every vertex $v$ in $A$ is at most one. Indeed, if it would be two then there would be two hyperedges with $v$ in their intersection such that $v$ is a tail-vertex in both of them, also by linearity $v$ would be their only intersection, contradicting the assumption of the theorem. Thus, the indegree of every vertex $v$ in $A$ is at least $deg_G(v)-1$. Thus, $b\le |A|=n$ and $a\ge \sum_k(k-1)n_k$. By the assumption, at each vertex in $B$ the indegree minus the outdegree is at least $1$, summing this over every vertex of $B$ we get that $|B|\le b-a$. 
	
	Putting these three inequalities together we get that $1\le |B|\le b-a\le n-\sum_k (k-1)n_k=\sum_k n_k-\sum_k (k-1)n_k=\sum_k(2-k)n_k\le n_1$. That is, there is a vertex with degree one, as required.	
\end{proof}

\begin{proof}[Proof of Theorem \ref{thm:21ex}]	
	First we prove that $ex(n,\FF)\ge 2\binom{n}{3}$.
	Let $\HH$ be the non-oriented hypergraph with vertex set $\{1,2,3,\dots, n\}$ which contains for all triples $i<j<k$ the hyperedges $(ik,j)$ and $(jk,i)$. The important property of $H$ is that for an arbitrary edge with tail vertices $a,b$ and head vertex $c$ it always holds that $c<\max(a,b)$. This hypergraph has $2\binom{n}{3}$ hyperedges so it remains to prove that $\HH$ does not contain $\FF$.
	
	Suppose on the contrary that $\HH$ contains $\FF$, let a copy of $\FF$ be on the vertices $i_1,i_2,i_3,i_4,i_5$ (where $i_j$ corresponds to vertex $j$ of $\FF$). 
	Now using the above observation to each edge of this copy of $\FF$ we get that 
	$$i_3<\max(i_1,i_2),$$
	$$i_4<\max(i_1,i_3)\le \max(i_1,i_2),$$
	$$i_5<\max(i_2,i_3)\le \max(i_1,i_2),$$
	$$i_2<\max(i_1,i_4)\le \max(i_1,i_2),$$
	$$i_1<\max(i_2,i_5)\le \max(i_1,i_2).$$
	
	Thus both $i_1$ and $i_2$ are smaller than $\max(i_1,i_2)$, a contradiction.
	
	We note that e.g. $F=\{(12,3),(13,4),(23,5),(34,2),(35,1)\}$ would be a good example as well. 
		
	To see the other statement of the theorem, first note that $ex_o(n,\FF)\le \binom{n}{3}$ is trivial as there are at most this many subsets of size $3$. Equality also follows by taking in the above $\HH$ for each triple of vertices only one of the two hyperedges whose vertex set is this triple. This gives an oriented hypergraph with $\binom{n}{3}$ hyperedges and of course it still does not contain $\FF$.
\end{proof}

\section{Proof of Theorem \ref{thm:3unif}}

We are given a $2\rightarrow 1$ hypergraph $\HH$ on vertex set $V$ with Property S. We want to show that it is proper $2$-colorable. $\HH$ may have multiple hyperedges on the same subset but with different special vertices, in which case we can just delete all but one of these hyperedges and a proper $2$-coloring of the new (non-oriented) hypergraph is also a proper $2$-coloring of the original. Thus from now on we assume that every $3$-subset of the vertices is the vertex set of at most one hyperedge of $\HH$.

We define a labeled graph $G$ on $V$ in which the labels are also the vertices of $\HH$, as follows. For each hyperedge $ab\rightarrow s$ of $\HH$ we assign with it an edge of $G$ on vertices $a,b$ with label $s$ (and there are no other edges in $G$). The label of an edge $e$ is denoted by $l(e)$. For a subset $S$ of $V$, $G[S]$ denotes the subgraph of $G$ on vertex set $S$ induced by the vertices of $S$.

Observe that no edge incident to some vertex $a$ can have label $a$.  Note that there may be parallel edges in $G$. We may refer to an edge on vertices $a,b$ as $ab$ when it does not lead to confusion (in particular, there are no other edges on the same vertices).

The next observation gives the property of $G$ corresponding to Property S of $\HH$ (they are in fact equivalent):

\begin{observation}\label{obs:2path}
	Given vertices $a,b,c$ and a 2-path in $G$ with edge $e$ on $a,b$ and $f$ on $b,c$. If $e$ and $f$ have different labels, then either $l(e)=c$ or $l(f)=a$ (or both).
\end{observation}

To this end let us call in a graph a 2-path on the vertices $(a,b,c)$ with its edges $e,f$ ($e$ on $a,b$ and $f$ on $b,c$) having different labels \emph{good}, if either $l(e)=c$ or $l(f)=a$ (or both), and \emph{bad} otherwise. Let us call a graph \emph{good} if all its 2-paths vertices with two differently labeled edges are good. Using this notation, $G$ must be a good graph.

As we noted, there may be parallel edges between two vertices (which then must have different labels).  We first prove Theorem \ref{thm:3unif} in the special case when there are no parallel edges in $G$ (phrased later explicitly as Theorem \ref{thm:3unifweak}). Afterwards, we will extend our proof also to the case when $G$ has parallel edges.

Thus from now on we assume that there are no parallel edges until stated otherwise, and thus $G$ is a simple graph.

\begin{lemma}\label{lem:x} If there are no parallel edges in $G$ then the following are true:
	
	\begin{itemize}
		\item[(i)] No vertex is incident to edges of $4$ different labels.
		\item[(ii)] If a vertex $x$ is incident to edges with $3$ different labels then there are exactly $3$ incident edges and they must have the following form: $xa$ with label $b$, $xb$ with label $c$ and $xc$ with label $a$.
		\item[(iii)] If a vertex $x$ is incident to edges of $2$ different labels then one of these two labels appears on only one edge incident to $x$ and the other endvertex of this edge is the other label that appears on the edges incident to $x$.
	\end{itemize}
\end{lemma}

\begin{proof}
		Assume first that $x$ is incident to $3$ edges with $3$ different labels (there might be further edges incident to $x$). Let the other endvertices of these edges be $a,b,c$. Using that the path $axb$ is good we get that $l(ax)=b$ or $l(bx)=a$. We assume $l(ax)=b$, as the other case can be finished similarly. Using that the path $axc$ is good we get that $l(ax)=c$ or $l(cx)=a$, as the first does not hold, we must have $l(cx)=a$. Finally, using that the path $bxc$ is good we get that $l(bx)=c$ or $l(cx)=b$. As the second does not hold, we must have $l(bx)=c$.
		
		Now assume that $x$ is incident to another edge, $xd$. Using that the paths $dxa$ and $dxb$ are good implies that $dx$ must have label $a$ and $b$ as well, a contradiction. Thus, there cannot be further incident edges, proving $(i)$ and $(ii)$.
		
		To show $(iii)$ assume first that there is a vertex $x$ such that $x$ is incident to edges $e_1,e_2,f_1,f_2$ such that $l(e_1)=l(e_2)=a$ and $l(f_1)=l(f_2)=b$. Then there is an index $i$ and $j$ such that $e_i$ is not incident to $b$ and $f_j$ is not incident to $a$. The path $e_if_j$ is bad, a contradiction. 
		Now assume that $x$ is incident to edges $e_1,e_2,f_1$ such that $l(e_1)=l(e_2)=a$ and $l(f_1)=b$. Using that $e_1f_1$ and $e_2f_1$ are both good we get that $f_1$ must have endvertices $x$ and $a$, as claimed.
		Finally, if $x$ is incident only to two edges and they have different labels then, using that they form a good path, one of them must have the label of the end of the other one, as required.
\end{proof}

\begin{figure}
	\begin{center}
		\includegraphics[width=6cm]{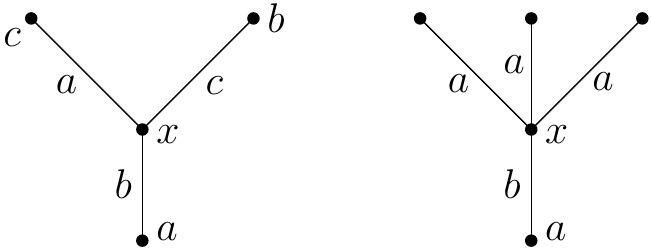}
		\caption{The possible structures guaranteed by Lemma \ref{lem:x}.}
		\label{fig:abc}
	\end{center}
\end{figure} 

For a vertex $q$ we define $K(q)$ (the \emph{core} or kin of $q$) to be the set of vertices that have at least $2$ incident edges with label $q$ (note that $K(q)$ does not contain $q$ and can be empty). Lemma \ref{lem:x}(iii) implies that $K(q)\cap K(r)=\emptyset$ for $q\ne r$. Let $K$ be the union of the cores and let $R$ (the residual vertices) be the vertices of $G$ that are not in any of the $K(q)$'s, that is, $R=V\setminus K$. 

First we give an intuition of our coloring strategy. In a $2$-coloring of the vertices a hyperedge $ab\rightarrow c$ is non-monochromatic either because the associated edge $ab$ in $G$ is non-monochromatic or because one of $a$ and $b$ gets a color different from $c=l(ab)$. We will color the vertices in $R$ such that most of the edges in $G[R]$ are non-monochromatic, thus corresponding hyperedges are non-monochromatic for the first reason. On the other hand we will color the vertices in $K$ such that for any vertex $x$ all but at most one of the vertices in $K(x)$ get a color different from $x$ which guarantees that most of the hyperedges corresponding to edges incident to some $K(x)$ are non-monochromatic for the second reason. 

\begin{observation}\label{obs:cq}
	A core $K(r)$ induces only edges with label $r$.	
\end{observation}
\begin{proof}
	Assume on the contrary that there is an edge with label $q$ with both endvertices in $K(r)$. Applying Lemma \ref{lem:x}(iii) on the two endvertices of this edge implies that both endvertices must be equal to $r$, a contradiction.
\end{proof}

Now we proceed by understanding the structure of $G[R]$. Notice that by definition of $R$, in $G[R]$ all pairs of edges forming a 2-path have different labels and so they must form a good path.

We partially direct the edges of $G$. An edge $ab$ is directed towards $a$ (and away from $b$) if $a$ is adjacent to the vertex $l(ab)$ in $G$. There can be undirected edges and edges directed in both ways. If an edge $ab$ is directed towards $a$ but not towards $b$ then we say that it is directed only towards $a$.

\begin{lemma}\label{lem:gr}
	 If there are no parallel edges in $G$ then the following are true:
	\begin{itemize}
		\item[(i)] For any vertex $x\in R$ and two edges incident to $x$ in $G$, at least one of these edges must be directed towards $x$. That is, at most one edge indicent to $x$ in $G$ is not directed towards $x$.
		\item[(ii)] If an edge $ab\in G[R]$ is directed in both ways then the vertex $c=l(ab)$ is not in $R$ and $l(ac)=l(bc)$ is a label different from $a$ and $b$.		
		\item[(iii)] 
		If a vertex $x\in R$ is incident to some edge $xb$ in $G[R]$ directed not only towards $x$ then there can be at most one further edge in $G[R]$ incident to $x$ and this edge (if it exists) must have label $b$ and directed only towards $x$.
	\end{itemize}	
\end{lemma}

\begin{proof}
	Part (i) follows from the fact that two edges incident to $x$ must form a good path (as the edges incident to an $x\in R$ have different labels).

	For part (ii) $ab$ being directed in both ways means that both $a$ and $b$ are connected with $c=l(ab)$. If the labels of $ac$ and $bc$ differ then the path $acb$ must be good therefore either $l(ac)=b$ or $l(bc)=a$, in both cases there would be two hyperedges on these $3$ vertices, a contradiction. Thus $l(ac)=l(bc)$ in which case $c\in K=V\setminus R$ by definition of the cores. Also, $a\ne l(ac)=l(bc)\ne b$, as claimed.
	
	For part (iii) assume first that there is an edge $xb$ incident to $x$ in $G[R]$ directed both ways. Let $xd$ be another edge incident to $x$ in $G[R]$. By part (ii) $l(xb)\notin R$ and thus $l(xb)\ne d$ and as $dxb$ is a good path, we have that $l(xd)=b$ and so $xd$ is directed towards $x$. If it would be also directed away from $x$ then we could apply part (ii) to the edge $xd$ to conclude that $b\notin R$, a contradiction. 
	
	Assume next that $xb$ is undirected or directed only away from $x$ and that there is another edge $xd$ incident to $x$ in $G[R]$. As $dxb$ is a good path and $xb$ is not directed towards $x$, we have $l(xd)=b$ and $xd$ is directed towards $x$. If $xd$ would be also directed away from $x$ then we could apply the previous case of part (iii) to $xd$ to conclude that $xb$ must be directed only towards $x$, a contradiction.
	
	So far we have seen that all edges incident to $x$ except for $xb$ are directed only towards $x$ and have label $b$. This implies that there must be at most one of these as $x\in R$ cannot have two incident edges with the same label. This finishes the proof of part (iii).	
\end{proof}

\begin{claim}\label{claim:grstructure}
	A connected component of $G[R]$ is either a cycle of length at least $4$ directed cyclically or a tree. If it is a tree then it contains at most two vertices, called the \emph{central vertices}, with the following properties:
	\begin{itemize}
		\item  If the component has one central vertex $x$ then $x$ has degree at most $3$ and every edge of the component incident to $x$ is directed only towards $x$.
		\item If the component has two central vertices then they are connected by an edge, called the \emph{central edge}, which is undirected or directed both ways and the rest of the edges of the component incident to the central vertices is directed only towards them.
		\item Every non-central vertex $y$ has degree at least one and at most two in the component: $y$ has one incident edge directed only away from $y$ and at most one more incident edge, which (if it exists) is directed only towards $y$.	
	\end{itemize}
\end{claim}  

\begin{figure}
	\begin{center}
		\includegraphics[width=16cm]{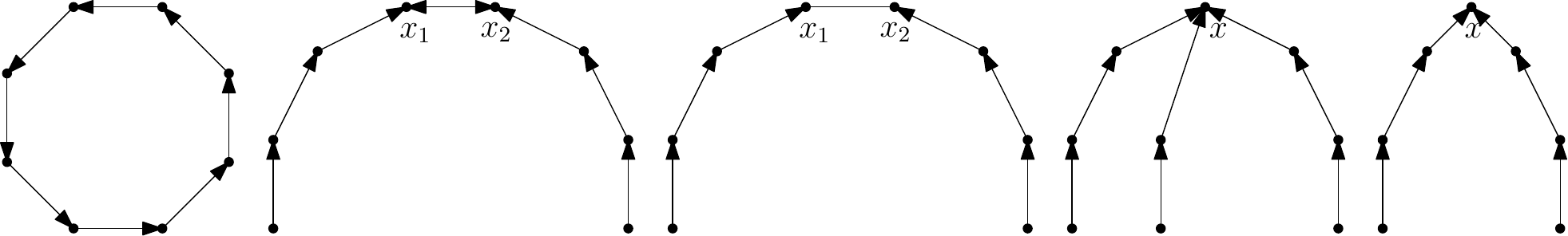}
		\caption{The possible structures of a component of $G[R]$ and its central vertices.}
		\label{fig:grstructure}
	\end{center}
\end{figure}

\begin{proof}
See Figure \ref{fig:grstructure} for how the components can look like.

Assume first that the component contains a cycle.

First we claim that there is no cycle of length $3$ (i.e., a triangle) in $G[R]$. Indeed, in such a cycle on vertices $a,b,c$ using that the path $abc$ is good we would get that one of the edges $ab$ and $bc$ is labeled with the third vertex and thus directed both ways and then by Lemma \ref{lem:gr}(ii) the third vertex is not in $R$, a contradiction.

Now let $C=c_0c_1\dots c_{k-1}$ be a cycle of length at least $4$ in $G[R]$. As any two adjacent edge in $G[R]$ forms a good path, $C$ must have an edge directed in some way, wlog. $c_0c_1$ is directed towards $c_0$. Then by Lemma \ref{lem:gr}(iii) $c_1c_2$ is directed only towards $c_1$ and $c_1$ has no other edge incident to it in $G[R]$. Continuing this argument with the edge $c_1c_2$ etc., we go around the cycle and see that all vertices have degree exactly two in $G[R]$ and they are directed only in one way such that together they form a cyclically directed cycle (in the last step we see that $c_0c_1$ cannot be directed both ways). Note that Lemma \ref{lem:gr}(iii) also implies that $l(c_ic_{i+1})=c_{i-1}$ for every $i$ where indices are modulo $k$.

Assume now that the component is a tree with at least one edge.

First, if there exists an edge $ab$ which is undirected or directed both ways in the component, then similar to the case of the cycle, we can apply Lemma \ref{lem:gr}(iii) to $ab$ to see that there can be at most one more edge $a_1a$ incident to $a$ and this is directed only towards $a$. Then we can apply Lemma \ref{lem:gr}(iii) repeatedly to the edge $a_1a$ to see that there is at most one more edge $a_2a_1$ incident to $a_1$ and this is directed only towards $a_1$, etc. to find a path of edges with $a$ being one of its endvertices in which all edges are directed towards their endvertex closer to $a$. The same way we can find another path of edges with $b$ being one of its endvertices in which all edges are directed towards their endvertex closer to $b$. The two paths together with $ab$ form the whole component, having the required properties with central vertices $a$ and $b$.

Second, if there is no undirected edge or an edge directed both ways then take a vertex $x$ for which every edge incident to $x$ in $G[R]$ is directed only towards $x$ (if there would be no such $x$ then starting from an arbitrary vertex we could find a directed cycle in this component contradicting that it is a tree). By Lemma \ref{lem:x} $x$ has degree at most $3$. We can again repeatedly apply Lemma \ref{lem:gr}(iii) for each incident edge to find at most three paths with one endvertex of them being $x$ such that every edge of every path is directed towards its vertex closer to $x$. The union of these at most $3$ paths have the required properties with central vertex $x$.
\end{proof}

Let us denote by $R_c$ the vertices of $R$ that are in a cycle-component in $G[R]$, and by $R_t$ the vertices that are in a path-component of $R$, note that $R=R_t\cup R_c$.

\smallskip
Now we proceed by understanding the structure of $K=V\setminus R$. We need to define an auxiliary directed graph $D$. We get the vertex set of $D$ from $V$ by contracting every $K(q)$ to a vertex $k(q)$ (plus we have the vertices of $R$) and we put an edge in $D$ directed from $k(q)$ to $q$ if $q\in R$ or directed from $k(q)$ to $k(r)$ if $q\in K(r)$. In this graph every vertex $k(q)$ has outdegree $1$ while vertices of $R$ have outdegree $0$. We note that there may be cycles of length $2$. This definition immediately implies the following:

\begin{claim}
	A connected component of $D$ is either a reverse rooted tree\footnote{A reverse rooted tree is a tree with one of its vertices being its root and in which every edge is directed towards this root vertex.} with its sink being its only vertex in $R$ or it contains exactly one cycle and is disjoint from $R$.
\end{claim}

We define a partition of $V\setminus R$. If a vertex $v$ of $G$ is in some $K(q)$ such that the component of $k(q)$ in $D$ has a cycle then we put $v$ into $K_c$, otherwise, if the component of $k(q)$ in $D$ is a tree, we put $v$ into $K_t$.

Thus, assuming that $G$ has no parallel edges, we have partitioned $V$ into $4$ parts, $K_c,K_t,R_c$ and $R_t$. We need an additional lemma about the structure of $G$ before we can start the coloring process.

\begin{lemma}\label{lem:nq}
	Given an edge $qx$ with label $r$ such that $q,r\in K$, $x\in R$. Let $X$ be the component of $G[R]$ containing $x$. Then either $X$ is a directed path whose sink and (the only) central vertex is $x$ or $X$ is a single edge $e$ with label $q$ directed in both ways such that $e$ is a central edge with central vertices $x,y$ such that $qy$ is an edge with label $r$.
\end{lemma}

\begin{proof}
	See Figure \ref{fig:coloring}, where the vertices in the grayed area can play the role of $x$.
	
	If the component of $x$ in $G[R]$ is an isolated vertex then we are done. Otherwise, take an edge $xy$ incident to $x$ in $G[R]$. If $l(xy)\ne q$ then as $r\notin R$, we have $l(qx)=r\ne y$ and so $qxb$ is a bad path, a contradiction. Thus $l(xy)=q$. As $x\in R$, it cannot have multiple incident edges with the same label, yet all incident edges have label $q$, therefore it has exactly one incident edge in $G[R]$ which has label $q$ and is directed towards $x$. Thus either $X$ is a directed path ending in $x$ and we are done or the edge $xy$ is a central edge directed in both ways. Then $y$ must be incident to the edge $yq$ and the assumptions of the lemma hold for $y$ (instead of $x$) as well. Thus, from the first part of the proof it follows that the only edge incident to $y$ in $G[R]$ is $xy$, finishing the proof.
\end{proof}

We are ready to define the $2$-coloring of $V$ which will be shown to be a proper $2$-coloring of $\HH$, see Figure \ref{fig:coloring} for an example.

\bigskip

\begin{figure}
	\begin{center}
		\includegraphics[width=16cm]{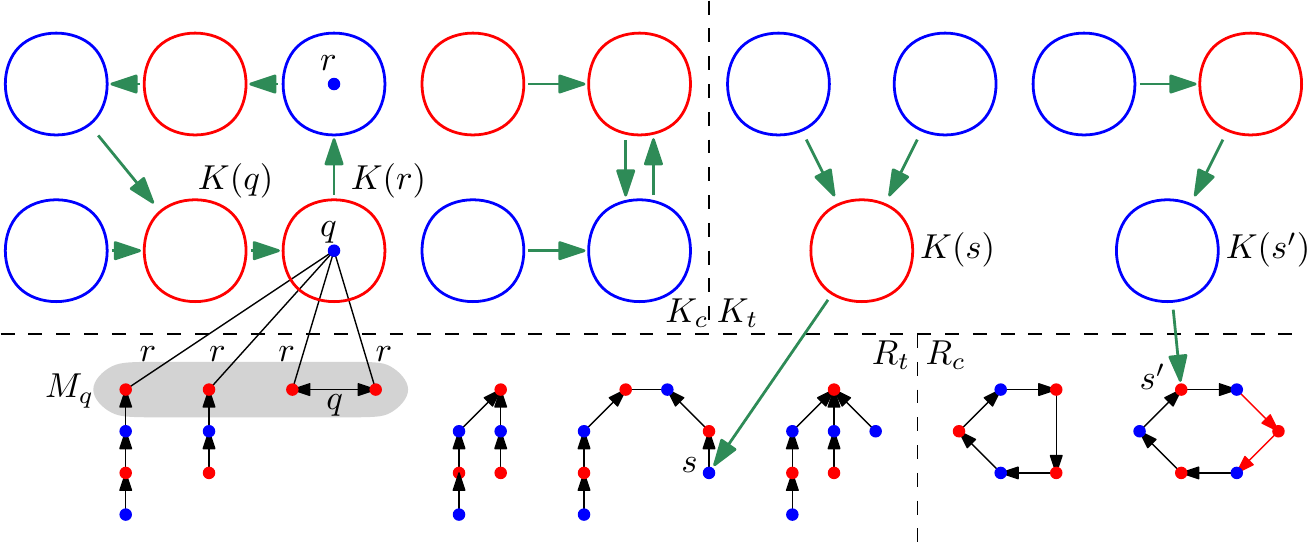}
		\caption{The structure of $G$ and the coloring the process gives. Edges with big arrows correspond to edges of $D$.}
		\label{fig:coloring}
	\end{center}
\end{figure}

\noindent{\bf The coloring process when there are no parallel edges} 

\begin{itemize}

\item[] {\bf Phase 1. Coloring the vertices of $K_c$.}

Recall that each component of $D$ in $K_c$ has exactly one cycle. If this cycle is even, that is, a component $D_1$ is a bipartite graph then we properly $2$-color it and then the vertices of $G$ contracted to some vertex of $D$ inherit the color of this vertex and we are done.
Otherwise, if the unique cycle is odd in a component $D_1$, then we select one edge $e_1$ of this cycle. Let $e_1$ be directed from $k(q)$ to $k(r)$. We temporarily delete $e_1$ from $D_1$, properly $2$-color the remaining tree graph $D_1'$ such that $k(q)$ and $k(r)$ are both colored red and then the vertices of $G$ contracted to some vertex of $D$ inherit the color of this vertex. Then we change the color of $q\in K(r)$ from red to blue. We call such a vertex $q$ as a \emph{rebel} vertex (note that there is at most one rebel vertex in each component of $D$). Let $M_q$, the set of \emph{minions} of the rebel vertex $q$, be the set of those vertices of $R$ which are connected to $q$ by an edge with label $r$ in $G$ (this $M_q$ may be empty). Lemma \ref{lem:nq} implies that the component $X$ of a minion vertex $x\in M_q$ in $G[R]$ is either a directed path ending in $x$ or it is a single edge $xy$ with $x,y\in M_q$, directed in both ways (see Figure \ref{fig:coloring}).

\item[] {\bf Phase 2. Coloring the vertices of $R_c$.}

We color the vertices of each cycle of $R_c$ such that the colors of the vertices on the cycle alternate cyclically except for at most one pair of consecutive vertices (this happens on cycles of odd size).
	
\item[] {\bf Phase 3. Coloring the vertices of $R_t$.}

First we color red all the central vertices in $R_t$ that are minions of some rebel vertex. Notice that by Lemma \ref{lem:nq} in each component of $R_t$ at most one vertex is colored red except when the component is a single central edge in which case both vertices may be colored red. Next we color the rest of the vertices of each tree component so that every edge of $G[R_t]$ is non-monochromatic except for those central edges which connect central vertices that are both minions of some rebel vertex (the same rebel vertex, by Lemma \ref{lem:nq}) and thus were already colored red.

\item[] {\bf Phase 4. Coloring the vertices of $K_t$.} 

Recall that every component of $D$ which has no cycle is a rooted tree with its sink being the only vertex in $R$. In every such component the root is already colored (in Phase 2 or 3) but the vertices contracted to the rest of the vertices are so far uncolored. Let us properly $2$-color this tree in $D$ starting from its root and then in $G$ the vertex inherits the color of the vertex of $D$ it was contracted to.

\item[] This finishes the coloring process.
\end{itemize}

\begin{claim}\label{claim:nopara2col}
	If $G$ has no parallel edges, the above defined coloring is a proper $2$-coloring of $\HH$.
\end{claim}	
\begin{proof}	
	
We need to check that for every edge of $G$ the corresponding hyperedge is properly colored. First we consider the edges of $G[R]$.

Most edges of the cycles of $R_c$ were properly colored, and so the corresponding hyperedges are also non-monochromatic. The only exceptions are singular edges in each odd cycle, but as the cycles are cyclically directed, the head-vertex of the corresponding hyperedge must be the next vertex on the cycle, which is colored differently from the endvertices of the vertex, and so the corresponding hyperedge is non-monochromatic.

Most edges of $R_t$ were properly colored, and so the corresponding hyperedges are also non-monochromatic. The only exceptions are those central edges that connect minion vertices. In this case by Lemma \ref{lem:nq} the label of the edge is the rebel vertex $q$ corresponding to its endvertices. Notice that $q$ got color blue, thus the hyperedge corresponding to this central edge is non-monochromatic (its tail is red and its head is blue).

We are left to consider edges incident to some vertex of $K$.

First, notice that in a $K(r)$ all but at most one of the vertices get a color different from the color of $r$. Therefore, if an edge lies inside some $K(r)$ then by Observation \ref{obs:2path} it must have label $r$ and so the corresponding hyperedge is not monochromatic (its head-vertex is differently colored from at least one of its tail-vertices). 

Further, for the same reason the edges with label $r$ incident to $K(r)$ are also corresponding to a non-monochromatic hyperedge with the only exception maybe when the edge $e$ is incident to a rebel vertex $q\in K(r)$. In this case, if the other endvertex of $e$ is in $R$ then it is in $M_q$ by definition and, recalling that all vertices of $M_q$ are colored red while $q$ is colored blue, this edge is properly colored and so the corresponding hyperedge as well. In this case $q$ is  a rebel tamed by its minions. Second, if the other endvertex is in some $K(s)$ then by Lemma \ref{lem:x}(iii) $s=q$ must be the case, but then in Phase 1 this endvertex, being in $K(q)$, is colored red while $q$ is colored blue, so the edge is properly colored and the corresponding hyperedge is also non-monochromatic. In this case $q$ is a rebel without a cause. 

We are left to consider edges incident to a vertex in $K(r)$ for some $r$ that  have a label different from $r$. There are two types of such edges. First, consider an edge with label $a\ne r$ connecting a vertex in $K(r)$ with a vertex in $R$. Notice that by Lemma \ref{lem:x} in this case the vertex in $R$ must be $r$ and so by the definition of the coloring in Phase 4 the two endvertices got a different color, and so the corresponding hyperedge is also non-monochromatic. Second, consider an edge with label $a\notin \{q,r\}$ which connects vertices in $K(r)$ and $K(q)$. In this case the two endvertices must be $q\in K(r)$ and $r\in K(q)$, corresponding to a $2$-cycle in $D$, thus the edge is properly colored in Phase 1, and so the corresponding hyperedge is also non-monochromatic.
\end{proof}

This finishes the proof in case $G$ has no parallel edges. Note that this already implies the following weakening of Theorem \ref{thm:3unif}: 

\begin{theorem}\label{thm:3unifweak}
Given a $2\rightarrow 1$ hypergraph $\HH$ s.t. for every pair of hyperedges $H_1,H_2\in \HH$, if $H_1\cap H_2\ne \emptyset$ then $H_1\cap H_2$ contains the head-vertex of at least one of these two hyperedges. Then $\HH$ admits a proper $2$-coloring.
\end{theorem}	

Next we prove Theorem \ref{thm:3unif} in full generality and thus from now on $G$ is allowed to have parallel edges. A pair of vertices $\{a,b\}$ is an \emph{overloaded pair} if there are parallel edges in $G$ between them. Notice that the edges on the same overloaded pair have different labels otherwise there would be multiple hyperedges in $\HH$ on the same vertex set.

\begin{figure}
	\begin{center}
		\includegraphics[width=12cm]{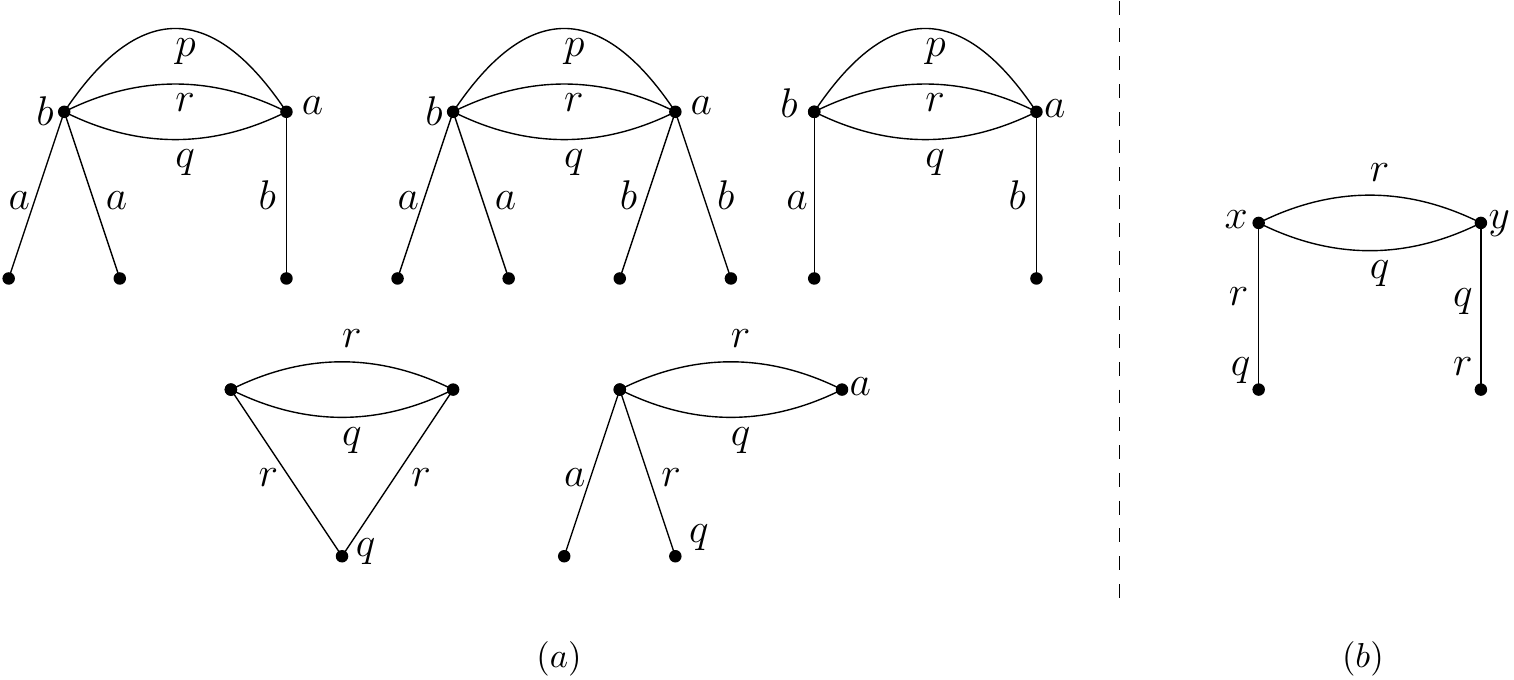}
		\caption{(a) Examples of weak overloaded pairs and (b) a strong overloaded pair.}
		\label{fig:overloaded}
	\end{center}
\end{figure}

There are two types of overloaded pairs. See Figure \ref{fig:overloaded}. An overloaded pair $\{x,y\}$ is {\em weak} if there exists an edge on this pair whose label is different from the label of every other edge incident to $x$ and $y$, in which case we choose one such edge and call it the \emph{representative edge} of the pair. An overloaded pair which is not weak is called {\em strong}.  

\begin{lemma}\label{lem:no2overloaded}
	No two overloaded pairs share a vertex.
\end{lemma}

\begin{proof}
	Assume on the contrary that $\{x,y\}$ and $\{y,z\}$ are both overloaded pairs. Assume there are edges with labels $r,q$ between $x,y$ witnessing that it is an overloaed pair. If two edges between $y,z$ have labels $r,q$ or labels completely different from $r,q$ then there is a path which is bad, a contradiction. The only remaining case wlog. is if the edges between $y,z$ have labels $r,s$ with $s\ne q$. Then we must have $s=x$ and $q=z$ to make all paths good. However, in this case there are two hyperedges on these three vertices, a contradiction.
\end{proof}

\begin{lemma}\label{lem:strong}
	If an overloaded pair $\{x,y\}$ is strong then there are exactly two parallel edges on this pair, with labels $r,q$ such that $x$ has an additional edge with label $r$ and $y$ has an additional edge with label $q$. Besides these, there are no other edges in $G$ incident to $x$ or $y$.
\end{lemma}

\begin{proof}
	By definition there must be at least two parallel edges on $\{x,y\}$. Let two of them have labels $r,q$. As the pair is not weak, there must be $1-1$ additional edge with label $r$ and label $q$ incident to $x,y$. If they would be incident to the same vertex, say $x$, then the corresponding hyperedges would be on the same triple of vertices, a contradiction. Otherwise, wlog. we can assume that $x$ (resp. $y$) has an additional incident edge with label $r$ (resp. $q$) whose other endvertex must be $q$ (resp. $r$) otherwise there would be a bad path. See Figure \ref{fig:overloaded}(b).
	
	We are left to prove that no other edge is incident to $x$ or $y$. Assume on the contrary that there is an additional edge. Wlog. we can assume this edge is incident to $x$ (and maybe also to $y$). It cannot have label $r$ as then it must end in $q$, but there is already such an edge. It also cannot have label $q$ as then its other endvertex must be $r$ and then there would be two hypergedges on the same triple of vertices. Finally, it cannot have label $s\ne p,q$. Indeed, first, if the other endvertex of the edge is $q$ or $r$ then $s=y$ otherwise there would be a bad path and then there are two hyperedges on the same triple of vertices, a contradiction. Second, if the other endvertex is not in $\{p,q\}$ then together with the edge $xq$ they form a bad path, a contradiction.
\end{proof}

For each weak overloaded pair we delete all edges between its vertices, except for the representative edge. For each strong overloaded pair we delete all edges between its vertices. We get a graph $G'$ without parallel edges. Now we are able to define with respect to this $G'$ the cores, $K$, $K_t$, $K_c$, $R$, $R_t$, $R_c$ and $D$.

Lemma \ref{lem:strong} implies the following:

\begin{corollary}
	The vertices $x,y$ of a strong overloaded pair are both in $R_t$.
\end{corollary}

Before defining the coloring we need to deal with the case handled in Lemma \ref{lem:nq} when strong overloaded pairs exist:

\begin{lemma}\label{lem:nqstrong}
	Given an edge $qx$ with label $r$ such that $q,r\in K$, $x\in R$. Let $X$ be the component of $G'[R]$ containing $x$. If $x$ participates in an overloaded pair $\{x,y\}$ then both $x$ and $y$ form single-vertex components in $G'[R]$.
\end{lemma}

\begin{proof}
	By Lemma \ref{lem:strong} the overloaded pair must have parallel edges with labels $r,q$ and the only edge incident to $x$ in $G'[R]$ is the edge $qx$. As $x\in K$, $x$ has no incident edges in $G'[R]$, therefore it is a single-vertex component of $G[R_t]$, as claimed. Similarly, $y$ has only one incident edge in $G'$, $yr$, with label $q$. As $r\in K$, $y$ has no incident edges in $G'[R]$, as required.
\end{proof}

\bigskip
\noindent {\bf The coloring process when there can be parallel edges} 

\begin{itemize}

\item[] {\bf Phase 1'. Coloring the vertices of $K_c$.} We color according to Phase 1 applied to $G'$.

\item[] {\bf Phase 2'. Coloring the vertices of $R_c$.} We color according to Phase 2 applied to $G'$. When doing this, for odd cycles we can arbitrarily choose the unique edge of the cycle which will be monochromatic. By Lemma \ref{lem:no2overloaded} we can choose this edge such that it is not on an overloaded pair of vertices.

\item[] {\bf Phase 3'. Coloring the vertices of $R_t$.} We color according to Phase 3 applied to $G'$ and the partial coloring we have so far, with the following exceptions. 

\begin{itemize}
	\item[(i)] 
	If in a strong pair $\{x,y\}$ the vertex $y$ is a single-vertex component of $G'[R]$ but the component of $x$ is not a single-vertex component, then by Lemma \ref{lem:nqstrong} $y$ is not a minion vertex and thus it is uncolored before Phase 3'. Therefore, when coloring the single-vertex component $\{y\}$, we can arbitrarily choose how to color $y$. We color $y$ to a color different from the color of $x$.
	\item[(ii)] 
	If in a strong pair $\{x,y\}$ both $x$ and $y$ are in a single-vertex 	 component of $G'[R]$, then at most one of them, wlog. $x$ can be a minion vertex. To see this, assume that one of them, wlog. $x$, is a minion vertex (and thus colored red in Phase 2'), then we claim that $y$ cannot be a minion vertex. Indeed, if $y$ would be also a minion vertex then by Lemma \ref{lem:nqstrong} $x$ would be a minion of the rebel vertex $q\in K(r)$ while $y$ would be a minion of $r\in K(q)$, that is, $K(r)$ and $K(q)$ would form a cycle of length two in $D$, contradicting that in $K_c$ only odd cycles contain rebel vertices. Thus $y$ is not a minion vertex and thus it is uncolored before Phase 3'. Therefore, when coloring the single-vertex component $\{y\}$, we can arbitrarily choose how to color $y$. We color $y$ to a color different from the color of $x$.
\end{itemize}

\item[] {\bf Phase 4'. Coloring the vertices of $K_t$.} We color according to Phase 4 applied to $G'$ and the partial coloring we have so far.

\item[] This finishes the coloring process.
\end{itemize}

\begin{claim} 
	The above defined coloring is a proper $2$-coloring of $\HH$.
\end{claim}	
\begin{proof}	
	Hyperedges corresponding to edges of $G'$ are non-monochromatic by Claim \ref{claim:nopara2col}. 
	
	Next we consider edges of $G\setminus G'$ that are on strong overloaded pairs. If an edge is on a strong overloaded pair $\{x,y\}$ then there are two cases (recall that $x,y\in R_t$). 
	
	First, assume that both of $x,y$ are in non-single-vertex components in $G'[R]$ (these two components may coincide). One of these components contains the edge $xq$, the other the edge $yr$, where $r,q$ are labels of the parallel edges on $x,y$. Then the edge $xq$ is non-monochromatic. Indeed, suppose on the contrary, then it must be a central edge with $x$ and $q$ both being minion vertices of some rebel vertex. Yet if $x$ is a minion vertex then by Lemma \ref{lem:nqstrong} it forms a single-vertex component, a contradiction. Thus $xq$ is non-monochromatic which implies that the hyperedge corresponding to the edge on $x,y$ with label $q$ is also non-monochromatic. Similarly, one can show that the edge $yr$ is non-monochromatic which implies that the hyperedge corresponding to the edge on $x,y$ with label $r$ is also non-monochromatic. 
	
	Second, if at least one of the components of $x$ and $y$ is a single vertex then in Phase 3' we guaranteed that they get different colors, and so the hyperedges corresponding to the parallel edges on $x,y$ are also non-monochromatic.

	We are left to consider edges of $G\setminus G'$ that are on weak overloaded pairs.
	
	For such an edge $e$ let $e'\in G'$ be the representative edge on the same pair of vertices. If $e'$ is properly colored during the coloring then $e$ is also properly colored and in turn the corresponding hyperedge as well. So we need to consider only those cases when $e$ might become monochromatic when coloring $G'$. 
	
	Notice that if one endpoint of $e'$ is in some $K_r$ then the label $a$ of $e'$ is different from $r$ by definition of a representative edge. Recall from the (last paragraph of the) proof of Claim \ref{claim:nopara2col} that in this case, i.e., when $e'$ is incident to some $K_r$ with label $a\ne r$, $e'$ (and in turn $e$) is non-monochromatic and we are done.
	
	Thus, we are left to deal with the case when $e'$ is in $R$. First, if $e'$ is in a cycle-component of $R_c$ then in Phase 2' we made sure that its endpoints get different colors and we are done. Second, if $e'$ is in a tree-component of $R_t$ then it is non-monochromatic unless $e'$ is a single-edge component in which the endpoints $x,y$ of $e'$ are minion vertices of some rebel vertex $q\in K(r)$. Then the label of $e$ must be $r$ otherwise we get a bad path on the vertices $q,x,y$. Notice that in Phase 1' $r$ (along with every vertex in the core in which $r$ is) gets color blue (see Figure \ref{fig:coloring}) and thus the hyperedge corresponding to $e$ is not monochromatic (its tail is red and its head is blue) and we are done.	
\end{proof}

We remark that the coloring process in the proof of Theorem \ref{thm:3unif} gives an efficient algorithm to properly $2$-color the vertices of a $2\rightarrow 1$ hypergraph having Property S. It can be easily implemented to run in $O(n^2)$ time, where $n$ is the number of vertices (recall that $\HH$ has $O(n^2)$ hyperedges and can have $\Theta(n^2)$ hyperedges), the details are left to the interested reader.

\bigskip
\noindent{\bf Acknowledgement}

The proof of Theorem \ref{thm:polyassim} is a joint work with D. P\'alv\"olgyi. The proof of Theorem \ref{thm:linear} is a joint work with P. \'Agoston, F. Bencs, Z. Bl\'azsik, G. Dam\'asdi, M. Nasz\'odi and B. Patk\'os done during the Tenth Eml\'ekt\'abla Workshop. The proof of Theorem \ref{thm:21ex} is a joint work with D. Gerbner. These results were not published before, they are part of this publication with their agreement, for which the author is grateful. Additionally, the author is thankful for these people and to A. Gyárfás, N. Salia and C. Xiao for the interesting discussions about this topic.

\bibliographystyle{plainurl}
\bibliography{dirhypergraphs}

\end{document}